\documentclass{amsart}

\usepackage{fullpage}
\usepackage{yhmath}  
\usepackage{amsmath,amsthm}
\usepackage{marvosym}
\usepackage{amssymb}
\usepackage[T1]{fontenc}
\usepackage[section]{algorithm}
\usepackage{algorithmic}

\usepackage{graphicx}
\usepackage{url}
\usepackage{color}
\usepackage{soul}
\usepackage{xfrac} 

\usepackage{booktabs}

\theoremstyle{plain}
\newtheorem{theorem}{Theorem}[section]

\newtheorem{cor}[theorem]{Corollary}

\newtheorem{proposition}[theorem]{Proposition}
\newtheorem{lemma}[theorem]{Lemma}

\newtheorem*{mainquestion*}{Main Question}

\newtheorem{dfn}[theorem]{Definition}
\newtheorem{conjecture}[theorem]{Conjecture} 

\newtheorem*{claim*}{Claim}

\theoremstyle{definition}
\newtheorem{defn}[theorem]{Definition}
\newtheorem{remark}[theorem]{Remark}
\newtheorem{example}[theorem]{Example}  

\newcommand{\Hth}{\mathbb{H}^3}
\newcommand{\Stw}{\mathbb{S}^2}
\newcommand{\Sth}{\mathbb{S}^3}
\newcommand{\Stwo}{\mathbb{S}^2}

\newcommand{\SN}{\mathbb{S}}
\newcommand{\SoneXStwo}{\SN^1 \times \SN^2}

\newcommand{\TT}{\mathcal{T}}

\newcommand{\Z}{\mathbb{Z}}
\newcommand{\CC}{\mathbb{C}}
\newcommand{\RR}{\mathbb{R}}

\newcommand{\Q}{\mathbb{Q}}

\newcommand{\PSLTC}{\text{PSL}(2, \mathbb{C})}

\newcommand{\PSLTF}{\text{PSL}(2, \mathbb{F})}

\newcommand{\Fp}{\mathbb{F}_p}

\newcommand{\PSLTFp}{\text{PSL}(2, \Fp)}
\newcommand{\PSLTFpSq}{\text{PSL}(2, \mathbb{F}_{p^2})}

\newcommand{\PSL}{\text{PSL}}

\newcommand{\PSLTR}{\text{PSL}(2, \mathbb{R})}

\newcommand{\SLTC}{\text{SL}(2, \mathbb{C})}

\newcommand{\mat}[4]{\left(\begin{matrix} #1 & #2 \\ #3 & #4 \end{matrix}\right)}

\definecolor{lesspsychedelicpurple}{rgb}{0.7, 0.1, .7}
\definecolor{bettergreen}{rgb}{0,0.6,0.4}
\definecolor{whatevs}{rgb}{0.1,0.7,0.4}

\begin{document}

\title{Small $\PSLTF$ representations of Seifert fiber space groups}
\dedicatory {In celebration of Alan Reid's 60th birthday.}
\author{Neil R Hoffman and Kathleen L Petersen}

\begin{abstract}
Let $M$ be a Seifert fiber space with non-abelian fundamental group and admitting a triangulation with $t$ tetrahedra. We show that there is a non-abelian $\PSLTF$ quotient where $|\mathbb F| < c(2^{20t}3^{120t})$ for an absolute constant $c>0$ and use this to show  that the lens space recognition problem lies in coNP for  Seifert fiber space input. We end with a discussion of our results in the context of distinguishing lens spaces from other $3$--manifolds more generally.  \end{abstract}

\maketitle

\section{Introduction}

Fundamental groups of surfaces and $3$--manifolds admit $\PSLTC$ representations that provide valuable tools for studying these low dimensional spaces. For example, hyperbolic 
$3$--manifolds admit discrete, faithful $\PSLTC$ representations. More generally, hyperbolic manifolds and many Seifert fiber spaces admit non-abelian (in particular non-trivial) 
representations into  $\PSLTC$.   This paper will focus on Seifert fiber spaces with infinite fundamental group.  We show that  these manifolds can be distinguished from the $3$--
sphere (and other lens spaces) via algorithms that also produce small certificates. The {\sc Lens space recognition}  problem is the problem of deciding if a given $3$--manifold is a 
lens space (including $\Sth$). A decision problem is said to lie in NP if an affirmative solution can be verified via certificate in polynomial time relative to the input size (of a triangulation 
in this case) and we say that a problem lies in coNP if a negative solution can be verified by such a certificate. That is,  the {\sc Lens space recognition} problem lies in coNP if given a 
manifold $M$ that is not a lens space, there is a certificate (for example, an explicit homomorphism to a non-cyclic group that can be written down from a triangulation of $M$) which 
can be checked in polynomial time. 

\begin{conjecture}\label{conj:s3andLpq}
 Let $M$ be a closed $3$--manifold.
\begin{enumerate}
\item {\sc $\mathbb{S}^3$ recognition} lies in coNP.
\item {\sc Lens space recognition} lies in coNP.
\end{enumerate}
\end{conjecture}

There are no non-orientable $3$--manifolds with finite fundamental groups.  
As we discuss in Section~\ref{sect:background}, there is a polynomial time algorithm (relative to the size of the triangulation) to determine if a triangulation represents an orientable or non-orientable $3$--manifold. Therefore, we can distinguish non-orientable $3$--manifolds from lens spaces (including $\Sth$) in polynomial time. 
  As a result, the subsequent arguments of this paper concentrate on the case where $M$ is a closed, connected and orientable $3$--manifold. Then weaving these two threads together, our main theorem addresses an important subclass of closed and connected $3$--manifolds, the Seifert fiber spaces.

\begin{theorem}\label{main_thm} For a Seifert fiber space $M$ with non-abelian fundamental group, the {\sc Lens space recognition}  problem lies in coNP.  In particular,  there is a polynomial time verifiable certificate to distinguish $M$ from $\Sth$.
\end{theorem}

As all of our certificates are either (non-trivial) non-abelian representations or non-cyclic abelian representations, we also distinguish these manifolds from $\SoneXStwo$ and so we state the following direct corollary.

\begin{cor}\label{cor:SoneXStwo} For a Seifert fiber space $M$ with non-abelian fundamental group, the {\sc $\SoneXStwo $ recognition}  problem lies in coNP.  
\end{cor}

The {\sc $\mathbb{S}^3$ recognition} problem lies in NP by work of Schleimer \cite{schleimer2011sphere}. This work was later extended by Lackenby and Schleimer  \cite{lackenby2022recognising} to show that the {\sc lens space recognition} problem lies in NP as part of a larger recognition scheme for manifolds with finite fundamental group. In light of  that result, determining if the problem also lies in coNP takes on special significance as it would provide another example of a decision problem in NP$\cap$coNP.

Zentner \cite[Theorem 11.2]{zentner2018integer} proved that  the {\sc $\mathbb{S}^3$ recognition}  problem is in coNP provided the Generalized Riemann Hypothesis (GRH) is true. His result is part of a larger investigation which exhibits non-trivial $\SLTC$ representations of integral homology spheres. After showing that any integral homology sphere with non-trivial fundamental group has a non-trivial $\SLTC$ representation \cite[Theorem 9.4]{zentner2018integer}, Zentner employs similar methods to those of Kuperberg \cite{kuperberg2014knottedness}. Conditional upon the truth of the GRH, Kuperberg established that {\sc unknot recognition} lies in coNP.  More recently, Lackenby \cite{lackenby2021efficient} removed the reliance on GRH and  proved that  {\sc unknot recognition} lies in coNP. It was previously known by work of Hass, Lagarias and Pippenger \cite{hass1999computational} that {\sc unknot recognition} lies in NP. Combining these results, it is now known that {\sc unknot recognition} lies in NP$\cap$coNP, and so it is natural to ask if {\sc unknot recognition} lies in P, the set of polynomial time algorithms. In service of this goal, Lackenby has also announced that there exists a sub-exponential time algorithm for {\sc unknot recognition}.

The goal this paper is to cut down the scope of the problem of distinguishing manifolds from lens spaces. By focusing on Seifert fiber spaces, we can remove the dependence on GRH, but not surprisingly, our methods do employ tools from number theory (Linnik's theorem). In particular, our main technique is to similar to those of Kuperberg and Zentner, in that we  show there is a sufficiently small finite  field $\mathbb F$ and a non-abelian representation of $\pi_1(M)$ into $\PSLTF$. These finite fields arise as quotients of number fields by a prime ideal.  We are able to establish unconditional results because the number fields we are working with are (nearly) cyclotomic, where current unconditional methods in number theory are sufficient.

\subsection{Organization and Outline} 

We now present a brief overview of our proof strategy. 
\begin{itemize}
\item[$\circ$] Using \cite{HarawayHoffman2019complexity} and an understanding of Seifert fiber spaces and their groups, we reduce the scope of the problem to distinguishing small, prime Seifert fiber spaces with non-cyclic fundamental groups from lens spaces.  In Section~\ref{sect:background} we provide background for computational complexity in this $3$--manifold setting and discuss distinguishing non-orientable manifolds and orientable manifolds. (As will be discussed below, the only non-prime Seifert fiber space is $\mathbb{RP}^3 \# \mathbb{RP}^3$. Since its fundamental group is isomorphic to $\Z/2\Z * \Z/2\Z$, it  surjects the dihedral group of order $6$ (also known as $\PSL(2,2))$. To satisfy the more general claim in the abstract, we point out that for any $c\geq 1$ and $t\geq 1$, $2 < c(2^{20t} 3^{120t})$.)  
\item[$\circ$]  In Section~\ref{sect:triangleReps}, we begin by discussing background information about Seifert fiber spaces. 
The small, prime, non-cyclic Seifert fiber space groups surject triangle groups $T_{n_1,n_2,n_3}$ where $n_k>1$ are integers for $k=1,2,3$.  We  demonstrate that for most of these triangle groups there is a particularly nice integral representation into $\mathrm{PSL}(2,K)$ where $K$ is a number field which is ``almost'' the cyclotomic field $\Q(\zeta_{2n_1n_2n_3})$.  The trace field of this representation has degree $\tfrac12\phi(2n_1n_2n_3)$. 
\item[$\circ$]  We then use Linnik's theorem to find a ``small'' prime that splits completely in $K$ and show that in the natural quotient the representation stays non-abelian. This gives us a non-abelian representation of these Seifert fiber space groups into $\mathrm{PSL}(2,\mathbb F)$ where $|\mathbb F|$ is bounded above by a polynomial function of $n_1n_2n_3$. (This covers most cases, and for the remaining cases we get a compatible bound.)  
\item[$\circ$]  Section~\ref{sect:degreebounds} deals with converting this upper bound to an upper bound in terms of $t$, the number of tetrahedra in a triangulation.  We accomplish this in two main steps.  
First, we show that for any $3$--manifold with a triangulation with $t$ tetrahedra, there is a presentation for $\pi_1(M)$  where the number of generators, relations, and their length is governed by $t$. 
Then we translate those complexity bounds for $\pi_1(M)$ into upper bounds for the degree of the trace field of a $0$-dimensional component of the $\PSLTC$ character variety in terms of $t$.  
\item[$\circ$]  In Section~\ref{sect:sfs_distinguish}, we reconcile these bounds. Our explicit representations of the triangle groups have trace fields of degree $\tfrac12\phi(2n_1n_2n_3)$ and we use this to translate our upper  bounds for $|\mathbb F|$ from a dependence on $n_1n_2n_3$ to a dependence on $t$, as needed. (This plan covers most cases, and we handle the remaining cases separately.)   We conclude with a discussion of the general problem of distinguishing $3$--manifolds from lens spaces. 
\end{itemize}


\subsection{Acknowledgements}{Both authors thank the Oklahoma State Number Theory group and members of the Oklahoma State, University of Arkansas and University of Oklahoma Topology groups for feedback on an early versions of this work. The first author is support by  Simons Foundation grant \#524123 to Neil Hoffman.} Both authors want to acknowledge the influence our common graduate advisor Alan Reid on the occasion of his 60th birthday. That this paper applies tools from number theory to topological problems is no accident. Moreover, we greatly appreciate his inspiration and continuing encouragement throughout our careers. Finally, we also thank him for feedback on early version of this paper.

\section{Background}\label{sect:background}


Our input is a triangulation.  Specifically,  a triangulation $T$ is a collection of $t$ tetrahedra and (exactly) $2t$ face pairings. For simplicity, we will assume a face pairing is a permutation in $S_4$, $0123 \rightarrow abcd$ with $a,b,c,d \in \{0,1,2,3\}$. 
Implicitly, we will enforce  the standard conditions necessary to ensure $T$ is the triangulation of a closed $3$--manifold (without boundary): each vertex link is a $2$--sphere, all edges are glued to themselves consistently, each face is glued to a distinct face, etc. To stress that we are using triangulated manifolds, we will often just refer our given triangulation as $M$ since  our methods  apply to any triangulation of the manifold.  
While there may exist slightly more efficient encoding schemes for triangulations, the scheme above is sufficient for the methods of this paper and the size of input needed to exhibit a triangulation is on the order of $t \log(t)$. 

Two tetrahedra $A$ and $B$ glued across a face have the same local orientation if the ordered set of edges $\{A(01),A(02),A(03)\}$ and the ordered set of edges $\{B(01),B(02),B(03)\}$ either both satisfy the right-hand or both satisfy the left-hand rule (after fixing an ordered basis on the tangent space for one point in the interior of $A$). This occurs if and only if the face pairing between $A$ and $B$ is orientation reversing, which occurs when the permutation lies in $S_4 \setminus A_4$.  We say a triangulation (and manifold) is \emph{orientable} if the local orientations agree for each tetrahedron in the triangulation and a triangulation is \emph{non-orientable} otherwise.

\subsection{Reduction to the orientable case}

We will now discuss orientation issues, so that we can safely ignore them for the rest of the paper.

\begin{proposition}\label{prop:nonorientability}
There is a polynomial $p$  such that given any triangulated $3$--manifold $M$ with $t$ tetrahedra in the triangulation, the number of steps needed to determine if the triangulation is orientable or non-orientable is at most $p(t)$.
\end{proposition}

The astute reader will notice that the solution presented below is not optimal in that we actually iterate over a tree twice. However, we still obtain a polynomial bound here.

\begin{proof}
Choose a maximal spanning tree $\Omega$ in the dual 1-skeleton $\hat{M}^{(1)}$. Since $\hat{M}^{(1)}$ has $2t$ edges and $t$ vertices, building this tree takes on the order of $t \log (t)$ steps using standard techniques in graph theory such as Kruskal's algorithm (see \cite{Kruskalspanningtrees} and \cite[Chapter 23]{cormen2022introduction} for further background).  Each generator $g_i$ of $\pi_1(M)$ corresponds to an edge $e_i$ in $\hat{M}^{(1)} \setminus \Omega$.  Choose a base point $x$ to be a vertex in $\Omega$ and record an orientation of the  corresponding tetrahedron in the triangulation by recording the labelling of the vertices. For each neighbor $y$ of $x$ in $\Omega$, we can choose an orientation reversing face pairing to induce a relabelling of the vertices of the tetrahedron corresponding to $y$. Repeat this process for adjacent vertices until all tetrahedra are labelled. Each edge $e_i$ in $\hat{M}^{(1)} \setminus \Omega$  corresponds to a generator of the fundamental group of $M$ since it extends uniquely to a loop in $\hat{M}^{(1)}$, and also carries with it an induced face pairing map.  The manifold $M$ is orientable if this induced face pairing is orientation reversing for all $i$ (which requires checking $t+1$ gluings) and non-orientable if for any $i$ the face pairing is orientation preserving (which requires checking at most $t+1$ gluings). That the number of steps is polynomially bounded (roughly on the order of $t \log(t)$) directly follows.
%
\end{proof}


Since there are no non-orientable $3$--manifolds with cyclic (or more generally finite) fundamental group, Proposition~\ref{prop:nonorientability} provides a polynomial time algorithm to differentiate non-orientable $3$--manifolds from lens spaces.

\subsection{Storing finite groups}\label{sub:storing_finite_groups}
We begin with some basic facts about the storing finite groups and the computations necessary to perform operations. A main focus of this paper is take an input (usually a triangulation with $t$ tetrahedra) and produce a polynomially-sized certificate (to show the manifold is not a lens space). 

Assume that $G$ is a finitely presented group and denote by $G^{ab}$ the abelianization of $G$ and $Tor(G^{ab})$ the torsion subgroup.  We begin by connecting the rank of $G^{ab}$ and order of $Tor(G^{ab})$ to the structure of a presentation for $G$.  Group presentations with 
relations of bounded length as below will naturally arise in the context of triangulations. We make the following observation from a 
first homology computation  using linear algebra.

Let $G$ be a finitely presented group with $g$ generators $x_k$ for $k=1, \dots, g$ and  $r$ relations $w_j=w_j(x_1, \dots, x_g)$ for $j=1,\dots, r$ each of  length at most $l$. Denote by $|w_j|_{x_k}$ the exponent sum of the generator $x_k$ in the relation $w_j$.  Create a homogeneous 
linear system of equations using $g$ variables $y_1, \dots,  y_g$ where the relation $w_j$  determines the equation $\sum_{i=1}^g |
w_j|_{x_i} y_i =0$. Let $A$ be the matrix with coefficients $a_{ij}=|w_j|_{x_i} $.   Denote by   $N$  the Smith normal form of $A$.  The abelianization of $G$ can be determined from    $N$  as follows. Let $N_{sq}$ be the square minor of $N$ such that  $n_{ii}$ is non-zero, so that  $rank(N) = rank(N_{sq})$. It follows that  $G^{ab} \cong \Z^{g-rank(N)} \oplus \Z/n_{ii}\Z$, where $\Z/1\Z$ represents the trivial group,   and 
\[ |Tor(G^{ab})|=\prod_{i=1}^{rank(N)} n_{ii} = \det(N_{sq}).\] 

We now bound  $\det(N_{sq})$. First let $A_{sq}$ be a square minor of $A$ with determinant of largest magnitude. As the computation of $N$ from $A$ only uses elementary matrices $|\det(N_{sq})|\leq |\det(A_{sq})|$.  Using Hadamard's inequality, we have that 
\[ |\det(N_{sq})|^2\leq |\det(A_{sq})|^2\leq \prod_{i=1} ^{rank(N)} \| a_j\|^2\] where $\|a_j\|$ is the norm of the jth row of  $A_{sq}$. Our assumption that the word lengths are bounded by $l$ implies that $ \| a_j\| \leq l$, so  that $\det(N_{sq})^2 \leq l^{2r}$.  
 
We now summarize  many of the key observations from the above argument:

\begin{proposition}\label{prop:homology}
 Let $G$ be a group with $g$ generators and $r$ relations each of length at most $l$. Then $G^{ab}$ has rank at most $g$ and $|Tor(G^{ab})|\leq l^r$.   
\end{proposition}


%


For our purposes, the  key takeaway from Proposition~\ref{prop:homology} is that the number of relations and their length determines a linear upper bound on $\log(|Tor(G^{ab})|)$ in terms of $r$, when $l$ is a constant. In particular, the bit-size needed store torsion elements in a homology group is bounded by a linear function. As noted above, the existence of free abelian quotients can also be quickly computed and exhibited. This  implies that we can verify that $G$ admits any abelian quotient that is necessarily a subgroup of $G^{ab}$ in polynomial-time.

We can also apply similar methods in accounting to representations into $\PSLTFp$ and $\PSLTFpSq$. If we bound $p$ by a polynomial, then we can perform the operations of $+,-,\cdot$ in $\Fp$ or $\mathbb F_{p^2}$ in polynomial-time.  This ensures  we can perform operations with (cosets of) $2\times 2$ matrices with entries in $\Fp$ or $\mathbb F_{p^2}$ in polynomial-time. As a result of Perelman's affirmative solution to the geometrization conjecture, lens spaces are the the only $3$--manifolds with finite cyclic fundamental group. Therefore, if $H_1(M)$ is not cyclic, we can distinguish $M$ from a lens space via a simple homology computation.

We now give an example of the kind of methods we will employ later in the paper.

\begin{example}
Consider the following presentation for the fundamental  group of the figure-8 knot complement $G= \langle a, b \mid  aba^{-1}b^{-1}a=baba^{-1}b^{-1} \rangle$.   Identify $\mathbb{F}_{25}$ with the quotient $ \mathbb{F}_5[x]/(x^2+1)$.  We define a homomorphism from $G$ to  $\text{PSL}(2,\mathbb{F}_{25})$ by 
\[ a \rightarrow  \pm \mat{x}{0}{0}{-x}, \quad b \rightarrow \pm  \mat{x}{-x}{0}{-x}.\]  
The image of $G$ is isomorphic to $D_{10}$, the dihedral group of order 10.

In order to produce a certificate that this is valid representation, we first need to be able to store elements in $\mathbb{F}_{25}$.  We write each element as a two dimensional vector $\mu_0+\mu_1x$ where  $\mu_0,\mu_1 \in \mathbb{F}_5$. It takes $\log_2(5-1)$ bits to encode $\{0,1,2,3,4\}$. Hence, encoding elements in  $\mathbb{F}_{25}$ takes twice that number of bits and encoding in $\mathbb{F}_{5^d}$ takes $d$ times the number of bits as it takes to encode an element in $\mathbb{F}_5$. More generally,  a $2\times 2$ matrix in  $\mathbb{F}_{p^d}$ takes $4d(\log_2(p-1))$ bits to encode. 

After encoding the matrices (which are the images of $a,b$ in $\mathbb{F}_{25}$), we then check that the relation is satisfied. In our case, the relation is length 10, so we need to preform the equivalent of $10$ matrix multiplications in $\text{PSL}(2,\mathbb{F}_{25})$ to check that this is a valid representation.  We point out that the number of generators and relations of the group need to have size bounded by our input as they are part of the certificate. 

We can also certify that this representation is non-abelian by observing that $ab$ and $ba$ have distinct images in $\text{PSL}(2,\mathbb{F}_{25})$.

In summary, the certificate for this representation is composed of the image of each one of the generators as a $2 \times 2$ matrix in $\mathbb{F}_{25}$ and a check that the relation (or more generally relations) of the group are satisfied by our choice of generators. We can certify that this, and any, representation is non-abelian by showing that the images of two specified generators do not commute. 
\end{example}

The natural generalization of the above example leads to the following proposition, which takes a finitely presented group as input. 
We say the {\em size} of presentation is the length of a string that includes all the generators and relations.

%
%


\begin{proposition}\label{prop:cert_non_abelian}
Let $G$ be a finitely presented group of size at most $n$ admitting a non-trivial representation, $\rho: G \rightarrow \PSLTFpSq$. There is a polynomial $f_1$ such that If $\log(p) < f_1(n)$, then there exists a second polynomial $f_2$, such that  the number of (bit-wise) computations needed to certify that $\rho(G)$ is a non-trivial subgroup of  $\PSLTFpSq$ is at most $f_2(n)$.  Moreover, if $\rho(G)$ is a non-abelian representation, then there is a third polynomial $f_3$, such that we can certify that $\rho(G)$ has a pair of non-commuting generators in at most $f_3(n)$ computations. 

\end{proposition}

In our applications, we will bound the number of generators and relations, while also keeping the relations of bounded size which implicitly bounds the size of a presentation for $G$. We also point out that a non-trivial $\PSLTFp$ representation determines a non-trivial $\PSLTFpSq$ representation by inclusion.

\section{Triangle groups and other quotients of $3$--manifold groups}\label{sect:triangleReps}

For the remainder of this paper, we assume that reader is familiar with $3$--manifold topology especially the taxonomy provided in \cite{scott1983geometries}.   We now provide definitions of the some of the relevant $3$--manifolds discussed in this paper.

A \emph{lens space} is a space obtained by gluing two solid tori $T_1,T_2$ along their boundaries such that the curve $p[\lambda_1]+q[\mu_1]$ ($p,q$ relatively prime) is identified with $[\mu_0]$ where $\mu_i$ is a curve that bounds a disk in the solid torus $T_i$ and $\lambda_1$ is isotopic to the core of $T_1$. In keeping with the standard conventions, we  denote this space by $L(p,q)$.  We keep the convention that a lens space has finite fundamental group  ie that $p\ne 0$. Our arguments also distinguish $\SoneXStwo$ from other other Seifert fiber spaces.
 We stress that $L(1,q) \cong \Sth$   is a valid lens space and our arguments that distinguish other manifolds from lens spaces  distinguish $\Sth$ from other manifolds as well.

 \begin{dfn}\label{def:triangle}
For $k=1,2,3$ let  $n_k>1$ be  integers. We define the triangle group
\[
T_{n_1,n_2,n_3} = \langle x, y \mid x^{n_1}, y^{n_2}, (xy)^{n_3} \rangle. 
\]
The triangle group $T_{n_1,n_2,n_3}$ is said to be hyperbolic if $\tfrac{1}{n_1}+\tfrac{1}{n_2}+\tfrac{1}{n_3}<1$, Euclidean if the sum equals $1$, and elliptic if the sum is greater than one. We will use the convention that $n_1\leq n_2\leq n_3$ and use $\ell$ to denote the quantity $2 \mathrm{lcm}(n_1,n_2,n_3)$. 
\end{dfn}

We refer the reader to \cite{scott1983geometries} and \cite{jaco1980lectures} for  background information on $3$--manifolds and Seifert fiber spaces in particular.
A Seifert fiber space $M$ is a $3$--manifold with a decomposition of $M$ into disjoint circles. In light of Proposition \ref{prop:nonorientability}, we will only consider orientable Seifert fiber spaces.   A $3$--manifold $M$ is \emph{prime} if for every $M\cong M_1 \# M_2$, either $M_1 \cong \Sth$ or $M_2 \cong \Sth$. The only orientable, non-prime Seifert fiber space is $\mathbb{RP}^3\# \mathbb{RP}^3$.    We say the Seifert fiber space $M$ is \emph{non-cyclic} if $\pi_1(M)$ is (non-trivial and) non-cyclic.   The set of orientable Seifert fiber spaces with cyclic fundamental group consists of  $\SoneXStwo$ and lens spaces. 
 In particular, if $M$ is non-cyclic, then $M \not\cong \Sth$. In summary, \emph{the set of orientable prime, non-cyclic Seifert fiber spaces includes all Seifert fiber spaces except $\mathbb{RP}^3\# \mathbb{RP}^3$, $\SoneXStwo$ and lens spaces}. Unless underscoring for clarity, we consider only orientable $M$ in our discussion of Seifert fiber spaces. 

The $3$--manifolds $\mathbb{RP}^3\# \mathbb{RP}^3$ and  $\SoneXStwo$ can be distinguished from lens spaces via homology in polynomial time, so our subsequent treatment in Section~\ref{sect:sfs_distinguish} to show that the lens space recognition problem is in coNP applies to these special cases. In addition, we can distinguish $\SoneXStwo$ and $\mathbb{RP}^3\# \mathbb{RP}^3$ from prime, non-cyclic Seifert fiber spaces and from one another via polynomial time verifiable certificates (see \cite[Theorem 3]{ivanov2008}).

A $3$--manifold is called \emph{small} if it does not contain a closed essential surface.  Any $3$--manifold that is not small can be distinguished from a lens space using the essential surface as the certificate, which can be verified in polynomial time as we discuss in Section~\ref{sect:sfs_distinguish}.    Therefore, we focus on small Seifert fiber spaces, which  can be characterized by their fundamental groups.  
{\em Elliptic} $3$--manifolds are those (orientable, geometric $3$--manifolds) which have finite fundamental group.  

Let $M$ be an orientable, small, prime, non-cyclic  Seifert fiber space with base orbifold  $B$.  Since $M$ is small and in particular atoroidal, it follows from Jaco and Shalen's work that $B$ is either an $\Stw$ with  three cone points  or $B$ is $\mathbb{RP}^2$ with one cone point (see \cite[IV.2.5 and Lemma IV.2.6]{JacoShalen1979}). In the second case, $M$ also admits a Seifert fibration over $S^2(2,2,p)$  (for example see \cite[p. 356]{scott1983geometries}).   To see that $M$ being small implies the bound on exceptional fibers, assume the genus is positive or there are more than three cone points.  Then there is an embedded curve in the base orbifold such that it and all of its powers are homotopically non-trivial. (In the case that the underlying space of $B$ is $\Stwo$, this curve is separating and has a least two cone points on either side of it.)  The lift of this curve into $M$ does not intersect the exceptional fibers, ensuring that it is an incompressible torus.  As such $M$ would not be small. Since $M$ is orientable, if the base orbifold were $\Sth$ with 0, 1, or 2 cone points then $M$ would be either a lens space or  $\SoneXStwo$.  Therefore, $M$ has  base orbifold of the form $B=\Stw(n_1,n_2,n_3)$ for  integers $n_k\geq 2$.
The fundamental group $\pi_1(M)$ surjects $\pi_1^{orb}(B)$ (see \cite[Lemma 3.2]{scott1983geometries} for example). Since $\pi_1^{orb}(\Stw(n_1,n_2,n_3)) \cong T_{n_1,n_2,n_3}$, we conclude that $\pi_1(M)$ surjects $T_{n_1,n_2,n_3}$.

For a small, prime, non-cyclic Seifert fiber space $M$, if $M$ is elliptic,  the base orbifold  $B$  is elliptic.  
If $M$ is not elliptic, then $B$ is  hyperbolic  or Euclidean.  This coincides with the terminology in Definition~\ref{def:triangle} for the associated triangle groups.  We now make this connection explicit (see \cite{scott1983geometries}, also  \cite{scott1978Torus} for further background).
The elliptic triangle groups are the finite triangle groups.  These are the $T_{n_1,n_2,n_3}$  associated to triples $(n_1,n_2,n_3)$ equal to  $(2,3,3),(2,3,4),$  $(2,3,5)$, and $(2,2,m)$ for $m\geq 2$.  
 If $M$ is a non-cyclic elliptic Seifert fiber space, then $\pi_1(M)$ surjects (at least) one of these elliptic triangle groups.

The Euclidean and hyperbolic triangle groups are infinite, and so $B$ and also $M$ are infinite as well.  
The Euclidean triangle groups are those associated to the triples $(2,4,4)$, $(2,3,6)$, and $(3,3,3)$ and manifolds with Euclidean or Nil geometry surject these triangle groups. Finally, if the triangle group is hyperbolic, the Seifert fiber space  coincides with $\widetilde{\PSLTR}$ or $\mathbb H^2 \times \RR$ geometries.

%
%
%
%
%


This discussion can be summarized by the following proposition.

\begin{proposition}\label{prop:surjectstriangle}
Let $M$ be an orientable, prime, non-cyclic Seifert fiber space.  Then either 
$M$ contains an embedded, essential torus or $M$ is a small Seifert fiber space  and $\pi_1(M)$  surjects a 
 triangle group.  
\end{proposition}

 As  will be discussed later, this allows us to consider three cases of prime, non-cyclic Seifert fiber spaces: $M$ is not small, $\pi_1(M)$ surjects a hyperbolic triangle group, and $\pi_1(M)$ surjects a non-hyperbolic triangle group.   We first focus on the case when $M$ is a small Seifert fiber space, specifically when $\pi_1(M)$ surjects a hyperbolic triangle group.  We will  reconcile the remaining cases later in Proposition \ref{prop:non_abelian}.

\subsection{$\PSLTR$ Representations of Triangle Groups}
 We now give an explicit  integral $\PSLTR$ representation of a hyperbolic triangle group.   A feature of this representation is that it is contained in a number field of minimal or near minimal possible degree.

For a set of integers $S = \{s_i\}_{i\in I}$, we say the \emph{greatest common divisor of $S$}, $\mathrm{gcd}(S)$, is the greatest integer $d$ such that $d | s_i$ for all  $i \in I$. For a fixed integer $n\geq 1$  let $\zeta_{n}=exp(2\pi i/n)$ and define $c_{n}=\cos(2\pi/n)$  so that $2c_{n} = \zeta_{n}+\zeta_{n}^{-1}$.  We will use values of the form $\cos(2\pi/2m)$ instead of $\cos(2\pi/m)$ in our construction as it would lead to an element of order 2 being central otherwise.   

We now prove a lemma which will be useful later. 
\begin{lemma}\label{lemma:Tsurjects}
If  $d_k$ divides $n_k$ for $k=1,2,3$ then $T_{n_1,n_2,n_3}$ surjects $T_{d_1,d_2,d_3}$.  In particular,  $T_{n_1,n_2,n_3}$ surjects $\Z/d\Z \times \Z/d \Z$ where $d=\gcd(n_1,n_2,n_3)$.
\end{lemma}

\begin{proof}
The surjection from $T_{n_1,n_2,n_3} $ as above to $T_{d_1,d_2,d_3} = \langle \xi, \eta \mid \xi^{d_1}, \eta^{d_2}, (\xi \eta)^{d_3} \rangle$ is given by $x\mapsto \xi$ and $y \mapsto \eta$.  Therefore, $T_{n_1,n_2,n_3}$ surjects $T_{d,d,d}$ whose abelianization is $\Z/d\Z \times \Z/d \Z$.
\end{proof}

We now collect some useful concepts.   Let $G$ be a subgroup of $ \PSLTC$ (or $\SLTC$).   We call a number field $K$  the {\em field of definition}  of $G$ if $G < \text{PSL}(2,K)$  and there is no number field $L\subset K$ such that $G < \text{PSL}(2,L)$. The {\em trace field} of   $G$  is  $\Q( \text{tr}(g): g \in G)$ and is a number field  in many known cases including when $G$ is (the image of) a triangle group, and when $\Hth/G$ is a finite volume $3$--manifold.  If the number field $K$ is the field of definition of $G$, then the trace field is a subfield of $K$.   For any number field $K$, we use $\mathcal O_K$ to denote the ring of integers in $K$.  We say the group $G$ is {\em rigid}  (in $\PSLTC$) if all discrete and faithful representations of $G$ into $\PSLTC$ are conjugate (in $\text{Isom}(\Hth)$).   Therefore, for a rigid group the trace fields of all discrete and faithful representations are the same up to complex conjugation, and therefore identical for real fields. Alan Reid pointed out to us that representations similar to those given below were studied by Waterman and Maclachlan in \cite{MR766224}. The  difference between the constructions is that our $M_n$ is the inverse of their $A$ and $C$ and we choose specific conjugations so that  our matrix entries are algebraic integers.
 
\begin{defn}\label{defn:varrho}
Let $n_k\geq 2$ be integers for $k=1,2,3$ and let $r$ be a root of \[ r^2+ 2r(c_{2n_1}-c_{2n_2})+2(1-2c_{2n_1}c_{2n_2}-c_{2n_3}).\]  For any integer $n>1$ define
\[ 
M_n= \pm \mat{2c_{2n}}{1}{-1}{0}    \ \text{ and } \  T_r= \pm \mat{1}{r}{0}{1}. 
\]
Further, define $\varrho: T_{n_1,n_2,n_3} \rightarrow \PSLTC$ by  $x\mapsto M_{n_1} ,  \ y\mapsto T_r M_{n_2}T_r^{-1}$.
\end{defn}
We will show in Proposition~\ref{prop:trianglePSL}  that this is a well-defined $\PSLTR$ representation. 
Explicitly, this definition gives 
\[
\varrho(x) =  \pm \mat{2c_{2n_1}}{1}{-1}{0},  \quad \varrho(y) = \pm \mat{2c_{2n_2}-r}{r^2-2rc_{2n_2}+1 }{-1}{r},     \]
and 
\[
\varrho(xy) =  \pm \mat{-2rc_{2n_1}+4 c_{2n_1} c_{2n_2}-1}{2r^2c_{2n_1}-4rc_{2n_1}c_{2n_2}+ r+2c_{2n_1}}{r-2c_{2n_2}}{-r^2+2rc_{2n_2} -1}.
\]
To prove Proposition~\ref{prop:onlycyclicabelians}, and show that modulo a well-chosen prime ideal this is non-abelian the following will be useful as well, 
\[
\varrho(yx) = \pm \mat{4c_{2n_1}c_{2n_2}-2 r c_{2n_1} -r^2+2r c_{2n_2}-1}{2c_{2n_2}-r}{-2c_{2n_1}-r}{-1}.
\]

We now show that this is a discrete and faithful $\PSLTR$ representation  and we compute the field of definition and trace field  for this representation.

\begin{figure}
\begin{center}
\includegraphics[width=3in]{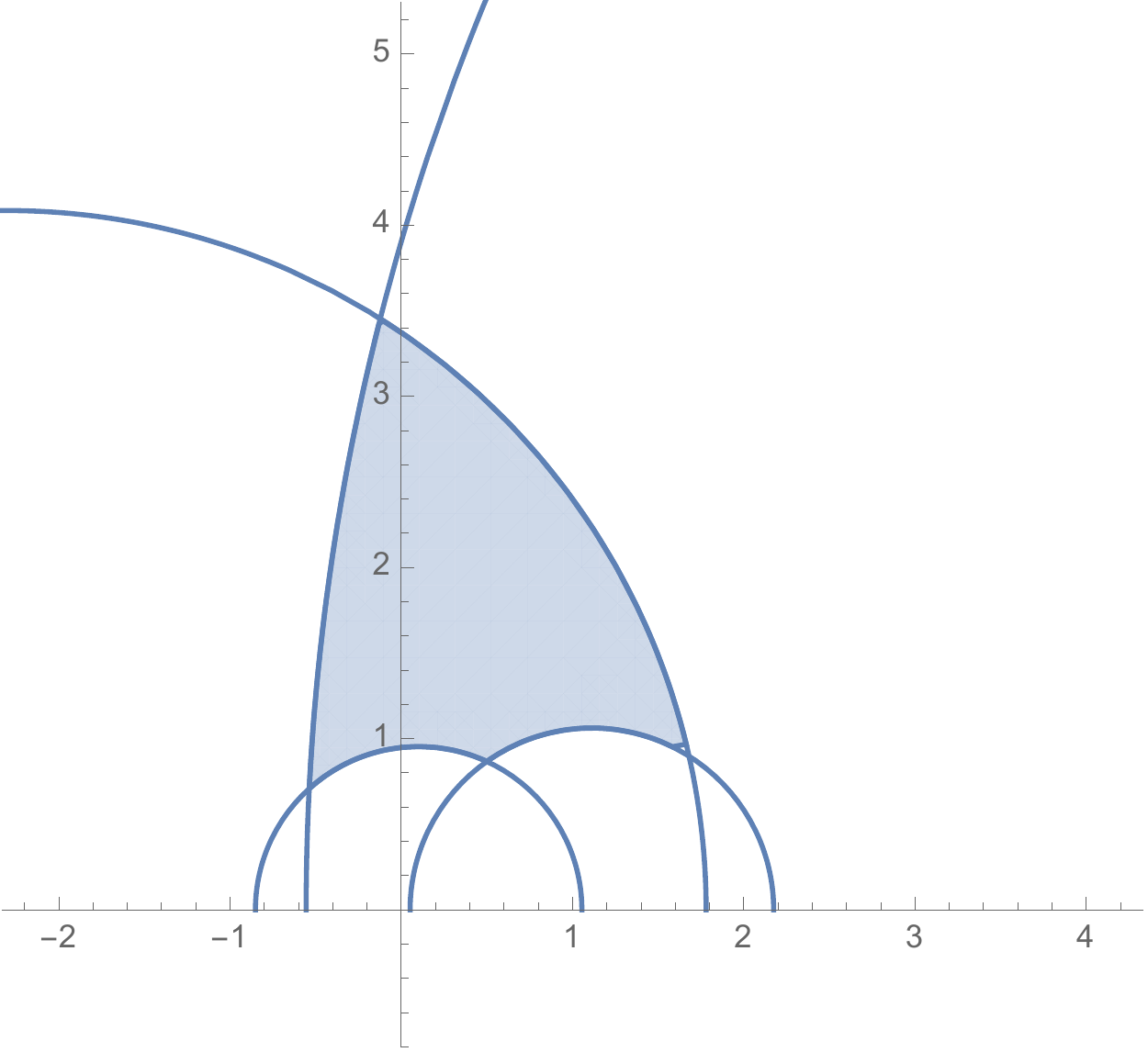}
\caption{\label{fig_triangle_proof} A fundamental domain for $T_{3,4,5}$ with the above construction.}
\end{center}
\end{figure}

\begin{proposition}\label{prop:trianglePSL} 
Let  $n_1\leq n_2 \leq n_3$ be positive integers with $\frac{1}{n_1}+\frac{1}{n_2}+\frac{1}{n_3} < 1$ and let  $\ell =2\thinspace\mathrm{lcm}(n_1,n_2,n_3).$
The following hold, where $K=\Q(c_{\ell},r)$  the totally real subfield of  $\Q(\zeta_{\ell}, r)$   
\begin{enumerate}
\item $\varrho$ is a rigid, discrete and  faithful representation of  $T_{n_1,n_2,n_3}$ into $\text{PSL}(2,\mathcal O_K)$.   
\item $K$ is the field of definition of $\varrho(T_{n_1,n_2,n_3})$. 
\item $\Q(c_{\ell})$ is the trace field of $\varrho(T_{n_1,n_2,n_3})$. 
\item $[K:\Q]$ is  $ \phi(\ell )/2$ or $\phi(\ell)$ depending on whether $r\in \Q(c_{\ell})$ or not.
\item The $\PSLTC$ character variety for $T_{n_1,n_2,n_3}$ is a finite set of points.
 \end{enumerate}
 
\end{proposition}

\begin{proof} Let $T=T_{n_1,n_2,n_3}$.
We begin by showing that $r$ as defined  is real. This occurs when the discriminant of the defining quadratic, as a function of $r$, is not negative.  Therefore,  it is enough to verify that
\[
4(c_{2n_1}-c_{2n_2})^2 - 4 (2-4c_{2n_1}c_{2n_2}- 2 c_{2n_3}) \geq 0.
\]
This is equivalent to showing that
\[
 (c_{2n_1}+c_{2n_2})^2+2c_{2n_3} \geq 2 \tag{$*$}.
\]
 Recall that  $n_1\leq n_2\leq n_3$ and $\tfrac{1}{n_1}+\tfrac{1}{n_2}+\tfrac{1}{n_3}<1$.   When $n_1\geq 3$ all cosines are at least $1/2$ and the left hand side of $(*)$ is at least 2 and so the expression is satisfied. Therefore it suffices to consider the case when $n_1=2$, so that $n_2 \geq 3$ by the hyperbolic condition.  If $n_2=3$ then $n_3\geq 7$ and $(*)$ holds as $\cos (2\pi / 2n_3)\geq 7/8$. Similarly, if $n_2=4$ then $n_3\geq 5$ and $(*)$ is satisfied as $\cos(2\pi/2n_3)\geq 3/4$.  Finally, if $n_2\geq 5$ then $n_3 \geq 5$,  so that $\cos (2\pi/10)>0.8$ and its square is at least $0.65$ therefore, $(c_{2n_1}+c_{2n_2})^2+2c_{2n_3} \geq 2.25$ as needed.  Therefore $r$ is real.

Next, we verify that the orders of $\varrho(x)$, $\varrho(y)$, and $\varrho(xy)$ are $n_1$, $n_2$, and $n_3$, respectively.  We use the fact that in $\PSLTC$, for an integer $n>0$  elements with  the following traces have order $n$:  $\pm 2 \cos(2\pi l /2n)$ where $\gcd(l,2n)=1$ and additionally for odd $n$,  $\pm 2 \cos(2\pi l/n)$  where $\gcd(l,n)=1$.   The fact that the order of $\varrho(x)$ is $n_1$ and the order of $\varrho(y)$ is $n_2$ follows from this.
The  trace of $\varrho(xy)$ is 
\[
\pm \big( -2+2r(c_{2n_2}-c_{2n_1})+4c_{2n_1}c_{2n_2}-r^2 \Big)
\]
and  $\varrho(xy)$ has order $n_3$ when this trace is $\pm 2 c_{2n_3}$.  Therefore, for $r$ as indicated the order of $\varrho(xy)$ is $n_3$.

The fundamental domain for $\rho(T)$ as pictured in Figure \ref{fig_triangle_proof} is consistent with the description of triangle groups described in \cite[IX.C]{maskit1988kleinian}, which is discrete, rigid, and faithful.  

All entries of $\varrho(x)$ and $\varrho(y)$ are integral and the minimal polynomial for $r$ over $\Q(c_{2n_1},c_{2n_2},c_{2n_3})=\Q(c_{\ell})$ is monic, so all entries of $\varrho(T)$ are integral. This shows that $\varrho(T) \subset \mathrm{PSL}(2,\mathcal O_K)$, completing the proof of  (1). 

Next, we will prove (2). This follows from the fact that  for either solution $r$ of its defining quadratic, the matrices defining $\varrho$  confirm that the entries of $\varrho(x)$ and $\varrho(y)$ are contained in   $\Q(c_{2n_1},c_{2n_2},c_{2n_3}, r)$ which equals $\Q(c_{\ell},r)$. It suffices to see that $K$ is the smallest field containing the entries of $\varrho(T)$. The values $\pm 2c_{2n_1}$, $\pm 2c_{2n_2}$ and $\pm 2c_{2n_3}$ are traces of $\varrho(x)$, $\varrho(y)$ and $\varrho(xy)$.  Therefore these values are contained in the field of definition.  Moreover, since $\pm(r-2c_{2n_2})$ is the $(2,1)$ entry of $\varrho(xy)$ and $2c_{2n_2}$ is contained in the field of definition, so is $r$. It follows that $K$ is the field of definition. 

The trace field is generated by the traces of $\varrho(x)$, $\varrho(y)$ and $\varrho(xy)$ (see \cite[Eqn. (3.25)]{MR03}).  As such,  the trace field for $ \varrho(T)$ equals  $\Q(c_{2n_1},c_{2n_2},c_{2n_3})=\Q(c_\ell)$, proving (3).  

The field  $\Q(c_{\ell})$ is the totally real subfield of $\Q(\zeta_{\ell})$, and so $K=\Q(c_{\ell},r)$ is the totally real subfield of $\Q(\zeta_{\ell},r)$.   As such, since  $[\Q(\zeta_{\ell}):\Q]=\phi({\ell})$ the degree   $[\Q(\zeta_{\ell},r):\Q]= \phi(\ell)$ or $2\phi({\ell})$ depending on whether $r\in \Q(\zeta_{\ell})$ or not.  Therefore,   $[K:\Q]= \phi({\ell})/2$ or $\phi(\ell)$ depending on whether $r\in \Q(c_{\ell})$ or not. This proves (4).

It remains to show (5). 
Let $\rho$ be a representation of $T$ into $\PSLTC$ with associated character $\chi_{\rho}: T \rightarrow \CC$ defined by $\chi_{\rho}(\gamma) = \text{tr}(\rho(\gamma))^2$ for all $\gamma \in T$. Since $T$  is a two generator group, the character variety is isomorphic to the set of characters $\{ \chi_{\rho}(x), \chi_{\rho}(y), \chi_{\rho}(xy)\}$.  The group relations for $T$ specify that the orders of $x,y$ and $xy$ are $n_1$, $n_2$, and $n_3$, respectively so that the orders of $\rho(x)$, $\rho(y)$ and $\rho(xy)$ must be  divisors of these $n_k$.
 Translating this as a trace condition, the character variety is determined by solutions to 
\[
\chi_{\rho}(x) = 4 \cos(2 \pi l_1/ 2n_1)^2, \ \chi_{\rho}(y) =   4 \cos (2 \pi l_2/ 2n_2)^2,  \ \chi_{\rho}(xy) = 4  \cos (2 \pi l_3/ 2n_3)^2
\]
  where $l_k$ divides $n_k$,  (when $n_k$ is odd  we also consider solutions equalling $4(\cos(2\pi l_k/n_k))^2$).
As such, it consists of a finite number of points. 
\end{proof}

\begin{remark}
It seems likely that $[\Q(c_\ell, r) : \Q(c_\ell)]=2$ for all choices of $n_1,n_2,n_3$. 
To show this index equals 2 it is sufficient to find a non-real embedding of $\Q(c_\ell, r)$.  Such an embedding exists if there is an integer $l$ with $\gcd(l, \ell)=1$ and 
\[
\big(\cos (2 \pi l/2n_2) + \cos (2 \pi l/2n_3)\big)^2 + 2 \cos (2 \pi l/2n_3) <2.
\]
A SageMath \cite{sagemath} computation (available as an ancillary file) confirms the fields are different for all of the 170 cases where $\tfrac1{n_1}+\tfrac1{n_2}+\tfrac1{n_3} <1$ and $n_i \leq 19$. In fact,  Waterman and Maclachlan point out in \cite[Theorem 3]{MR766224} that for very similar representations only at most finitely many triples can have the fields being equal.

\end{remark}

\begin{lemma}\label{lemma:cosinenorm}
Assume $n> 2$. 
\begin{enumerate}
\item $|N_{\Q(\zeta_{2n})/\Q}(2c_{2n})|$ equals $p^2$ if $n$ is twice a power of the prime $p$,  and 1 otherwise.
\item $|N_{\Q(\zeta_{2n})/\Q}(2c_{2n}-2)|$ equals $4$ if $n$ is a power of 2, and 1 otherwise. 
\end{enumerate}
\end{lemma}

\begin{proof}
 For an integer $k\geq 1$, let $\Phi_k$ denote the $k$th cyclotomic polynomial, and 
 define
\[
X_k=\prod (e^{ 2\pi i l /k}+1),   \quad  Y_k=\prod (e^{ 2\pi i l /k}-1) 
\]
 where the products are over all  integers $l$ between 1 and $k-1$ with $\gcd(l,k)=1$.  Therefore,  $X_k=\Phi_k(-1)$ and $Y_k = \pm \Phi_k(1)$.   By \cite{MR3842902}, $\Phi_k(-1)$ is $-2$ when $k=1$, $0$ when $k=2$, $p$ when $k=2p^e$ where $p$ is a prime, and 1 otherwise. Similarly,  $\Phi_k(1)$ equals  $1$ if $k$ is not a prime power and $p$ if $k$ is a power of the prime $p$  (see \cite{MR0282947} page 73).  Let $N$ denote the field norm $N_{\Q(\zeta_{2n})/\Q}$. The statement will follow upon showing that  $N(2c_{2n}-2) = Y_{2n}^2$, and $N(2c_{2n})= X_n^2$ if $n$ is even,  and equals $X_n$ if $n$ is odd.  Let $\zeta=\zeta_{2n}=e^{2\pi i/2n}$ and recall that   $2c_{2n} = \zeta+\zeta^{-1}$.

First we prove $(1)$. Since $\zeta$ is a unit of norm 1, $N (2c_{2n}) = N (\zeta^2+1)$ and 
 \[ N (\zeta^2+1) = \prod (e^{ 2(2\pi i l /2n)}+1)= \prod (e^{ 2\pi i l /n}+1) \]
 where the $l$ values are those integers   between 1 and $2n$ with $\gcd(l,2n)=1$. First, assume that $n$ is even so that $\gcd(l,2n)=1$ occurs exactly when $\gcd(l,n)=1$.  The $l$ values above are all integers $l'$ between 1 and $n$ with $l'$ relatively prime to $n$, and additionally all $l'+n$ values for these $l'$.  The terms for a $l'$ value and a $l'+n$ value are the same, so this norm equals  $X_n^2$.  When $n$ is odd, $\gcd(l,2n)=1$ occurs exactly when $\gcd(l,n)=1$ and $l$ is odd. A complete set of such $l$ modulo $2n$ reduces to all $l$ relatively prime to $n$ modulo $n$. Therefore, the norm is $X_n$ in this case.

Now we prove $(2)$. Consider $2c_{2n}-2 =  \zeta+\zeta^{-1}-2 = \zeta^{-1}(\zeta-1)^2$.  Since $\zeta$ is a unit of norm 1, $N (2c_{2n}-2) = N (\zeta-1)^2$. We now compute 
\[
N (\zeta-1) = \prod (e^{2\pi i l/2n} -1) 
\]
where the product is over all integers $l$ between 1 and $2n$ with $\gcd(l,2n)=1$. This is $Y_{2n}$ as needed.
\end{proof}

The following proposition uses the notation established in Definition~\ref{defn:varrho} and Proposition~\ref{prop:trianglePSL}.

\begin{proposition}\label{prop:onlycyclicabelians} Let $\pi$ be a prime ideal in  $\mathcal O_K$ and let $\bar{\varrho}$ denote the image of $\varrho$ under the reduction modulo $\pi$ map.  If $\gcd(n_1,n_2,n_3)=1$, $\tfrac1{n_1}+\tfrac1{n_2}+\tfrac1{n_3}<1$,  and $\pi$  lies over a rational prime $p \equiv 1 \pmod \ell$  then the image $\bar{\varrho}(T_{n_1,n_2,n_3})$ is non-abelian.
 \end{proposition}

\begin{proof}
Use $\equiv$ to denote reduction modulo $\pi$.  By the definition of $\ell$, it follows that for $k=1,2,3$, the integers $2$ and  $n_k$ are coprime to $p$ and therefore are units modulo $\pi$.   By Lemma~\ref{lemma:cosinenorm}, for $n>2$ the norm of $2c_{2n}$ is either $\pm 1$ or $\pm q^2$ where $\pm q^2$ occurs only when $2n$ is twice a power of the prime $q$. Therefore  any such $q$ is coprime to $p$ as well, and  $2c_{2n_k}$ is a unit in the quotient for $k=1,2,3$.   We will show that $\bar{\varrho}$ is non-abelian by showing that  $\bar{\varrho}(xy) \neq  \bar{\varrho}(yx)$.  As we are working in $ \PSLTC$ we will consider the $+$ and $-$ solution in terms of the (reduction of the) matrices given above.

We begin by considering the case when $n_1=2$ and assume that  $\bar{\varrho}(xy)\equiv \bar{\varrho}(yx)$.  First, consider the $-$ solution. From the $(2,1)$ entries of $\bar{\varrho}(xy)$ and $\bar{\varrho}(yx)$,  $2c_{2n_2} \equiv 0$, so that $r^2\equiv -2$ from the $(1,1)$ entries. The defining equation for $r$ reduces to  $r^2+2(1-c_{2n_3})\equiv 0$ and we conclude that $2c_{2n_3}\equiv 0$ as well.   As such, $p$ divides $N_{\Q(\zeta_{2n_2})/\Q}(2c_{n_2})$ and $N_{\Q(\zeta_{2n_3})/\Q}(2c_{n_3})$.  By Lemma~\ref{lemma:cosinenorm}, since $\pi$ and $n_2n_3$ are relatively prime, we conclude that $n_2=n_3=2$ so that $n_k=2$ for $k=1,2,3$ contradicting the assumption that $\gcd(n_1,n_2,n_3)=1$.  Next, consider the $+$ solution. The $(1,1)$ and $(1,2)$ entries of $\bar{\varrho}(xy)\equiv \bar{\varrho}(yx)$ imply that $2r\equiv 2c_{2n_2}$ and $r(r-2c_{2n_2})\equiv 0$. Since the quotient is a field  either $r\equiv 0$ or $r\equiv 2c_{2n_2}$.   In the second case, since $2r\equiv 2c_{2n_2}$ we conclude that $2c_{2n_2} \equiv r\equiv 0$.  In either case $r\equiv 0$ and therefore $2c_{2n_2} \equiv 0$ as well.   The defining equation for $r$ implies that $2-2c_{2n_3}\equiv 0$. By Lemma~\ref{lemma:cosinenorm} we conclude that $p=2$ which cannot occur.

Now we assume that $n_k>2$ for $k=1,2,3$. The image is non-trivial since $\bar{\varrho}(x)$ has off-diagonal entries equal to $\pm 1$.  Consider the $-$ solution.   By equating the $(2,1)$ entries, we have that $2c_{2n_2}\equiv -2c_{2n_1}$.  Using this to substitute for $2c_{2n_2}$ the $(1,2)$  entries imply that  $2 c_{2n_1} r( r+2c_{2n_1})\equiv 0$.  The quotient is a field.   It follows that  either $r\equiv 0$ or $r\equiv -2c_{2n_1}$.  If $r\equiv 0$ then the $(2,2)$ entries imply that $2\equiv 0$ and so $p=2$, which cannot occur.  If $r\equiv -2c_{2n_1}$ then the $(2,2)$ entries imply that $4c_{2n_1}^2- 1\equiv 0$. This factors as $(2c_{2n_1}+1)(2c_{2n_1}-1) \equiv 0$ so that $2c_{2n_1} \equiv \pm 1$ and $r\equiv \mp 1$. With this, the  $(2,2)$ entries imply that $1\equiv 0$ which cannot happen.

Consider the $+$ solution.  The $(2,2)$ entries imply that $ r( r-2c_{2n_2} ) \equiv 0$ and so $r\equiv 0$ or $r\equiv 2c_{2n_2}$.   First consider the case when $r\equiv 0$.  Then the $(2,1)$ entries imply that $2c_{2n_2} \equiv 2c_{2n_1}$.   Since $r\equiv 0$, the $T_r$ matrix is trivial and  $\bar{\varrho}(x)=\bar{\varrho}(y)$ and the image is cyclic.   Therefore, $\varrho(x)^{n_1}=1$, $\varrho(y)^{n_2}=\varrho(x)^{n_2}=1$ and $\varrho(xy)^{n_3}=\varrho(x)^{2n_3}=1$, and the order of the cyclic quotient divides $\gcd(n_1,n_2,2n_3)$.  By assumption $\gcd(n_1,n_2,n_3)=1$, so we need only consider when $2=\gcd(n_1,n_2,2n_3)$. Modulo $\pi$, this  cyclic group is generated by 
\[ \bar{\varrho}(x) \equiv \mat{2c_{2n_1}}{1}{-1}{0}.\] 
For this to have order 2, $2c_{2n_1}\equiv 0$ and so $p\mid N_{\Q(\zeta_{\ell})/\Q}(2c_{2n_1})$.  By Lemma~\ref{lemma:cosinenorm}, this norm is $\pm 1$ unless $n_1$ is prime, whence the norm is $\pm n_1^2$. By construction, $p$ is coprime to $2n_1$  giving a contradiction.   It remains to assume that  $r\equiv 2c_{2n_2}$.  The $(2,1)$ entries imply that $2c_{2n_2}  \equiv - 2c_{2n_1}$.   By substituting the $r$ and $2c_{2n_2}$ values in the minimal polynomial for $r$ with $-2c_{2n_1}$,  we see that $2c_{2n_3}-2\equiv 0$. Therefore $p$ divides $N_{\Q(\zeta_{2n_3})/\Q}(2c_{2n_3}-2)$. By Lemma~\ref{lemma:cosinenorm}, the  norm is $\pm 1$ if $n_3$ is not a power of 2 and is $\pm 4$ otherwise.  This implies that $p=2$ which is a contradiction.
\end{proof}

We now prove  that   there is a non-trivial representation of any triangle group $T_{n_1,n_2,n_3}$ into $\PSLTF$ for a small order finite field $\mathbb F$.  Here, we assume that $\gcd(n_1,n_2,n_3)=1$; by Lemma~\ref{lemma:Tsurjects} if $\gcd(n_1,n_2,n_3)=d$ then $T_{n_1,n_2,n_3}$ surjects the non-cyclic group $\Z/d\Z \times \Z/ d\Z$.

\begin{theorem}\label{thm:using_Linnik}
Let $n_k\geq 2$ be integers for $k=1,2,3$,  $\ell=2\mathrm{lcm}(n_1,n_2,n_3)$, and $\frac{1}{n_1}+\frac{1}{n_2}+\frac{1}{n_3} < 1$. There is an effectively computable $c>0$ and a field $\mathbb{F}$ with $| \mathbb F|  \leq c \ell^{10}$ such that  $T_{n_1,n_2,n_3}$ has a  non-trivial representation into $\PSLTF$. If $\gcd(n_1,n_2,n_3)=1$ then this representation is non-abelian. 
\end{theorem}

\begin{proof} Let $T=T_{n_1,n_2,n_3}$. 
By Xylouris's improvement upon Linnik's theorem, there is an effectively computable constant $c$ as above and prime $p\equiv 1 \pmod \ell$ with $p\leq c\ell^{5}$.  (This is proven in Theorem 2.1 of his doctoral thesis  \cite{MR3086819}.  In \cite{MR2825574} he demonstrates an exponent of 5.18.) The congruence condition on $p$ ensures that $p$ splits completely in $\Q(\zeta_{\ell}),$  and therefore $p$ splits completely in $\Q(c_{\ell})\subset \Q(\zeta_{\ell})$. (See \cite{MR0457396} Chapter 3 Corollary of Theorem 26 on page 78 for details about splitting and congruences in cyclotomic extensions.)  Let $r$ be as defined in Proposition~\ref{prop:trianglePSL}.  Since $K=\Q(c_{\ell},r)$ equals $\Q(c_{\ell})$ or is a quadratic extension of $\Q(c_{\ell})$ then any prime ideal $\pi$ in $\mathcal O_K$ lying over $p$ has $N_{K/\Q}(\pi)=p$ or $p^2$.  Proposition~\ref{prop:trianglePSL}  shows that $\varrho$ is a representation of $T$ into $\text{PSL}(2,\mathcal O_{K})$. Upon  composing with the reduction modulo $\pi$ map we have a homomorphism $\bar{\varrho}$ of $T$ into $\PSLTF$  where $\mathbb F= \mathcal O_{K}/\pi$ has order $p$ or $p^2$.  Putting this together,  $\bar{\varrho}(T)< \PSLTF$  where $|\mathbb F|\leq (c\ell^{5})^2$.

This representation is non-trivial since the $(1,2)$ and $(2,1)$ entries of $M_n$ are the units $\pm 1$  which survive the reduction  modulo $\pi$ mapping.
By Proposition~\ref{prop:onlycyclicabelians} the representation is non-abelian if $\gcd(n_1,n_2,n_3)=1$.
\end{proof}

A relatively straight-forward analysis handles the cases $\tfrac1{n_1}+\tfrac1{n_2}+\tfrac1{n_3} \geq  1$. Just as above, the following argument also produces small quotients. In this case, we can be more specific about these quotients. Namely, they are either $\Z/2\Z \times \Z/2\Z$,$\Z/3\Z \times \Z/3\Z$, or $\PSLTF$ with $|\mathbb F|\leq n_3^2$. The group $\mathbb Z/p \mathbb Z \times \mathbb Z/p \mathbb Z$  embeds in $\text{PSL}(2,\mathbb F_{p^2})$, so all of these images are contained in $\text{PSL}(2,\mathbb F)$ for  $|\mathbb F|\leq c \ell^{10}$ as in Theorem~\ref{thm:using_Linnik}.
We have chosen to remark on the non-cyclic abelian quotients separately because homology can be quickly computed directly.

\begin{proposition}\label{prop:non_abelian}
Let   $n_k\geq 2$ be integers for $k=1,2,3$ and    assume that $\tfrac1{n_1}+\tfrac1{n_2}+\tfrac1{n_3} \geq  1$. 
One of the following  holds
\begin{enumerate}
\item  $T_{n_1,n_2,n_3}$  surjects a non-cyclic abelian group of order at most 9, or 
\item $T_{n_1,n_2,n_3}$ has a non-abelian representation into  $\PSLTF$ with   $|\mathbb F| \leq n_3^2$.
\end{enumerate}
\end{proposition}

\begin{proof}

The possible triples are $(2,3,6)$, $(2,4,4)$, $(3,3,3)$, $(2,3,3)$, $(2,3,4)$, $(2,3,5)$ or $(2,2,m)$ with $m\geq 3$.

In the spherical cases,
\[
T_{2,3,3} 	\cong \text{PSL}(2,\mathbb F_3), T_{2,3,4}<S_4  \cong \text{PSL}(2,\mathbb F_9), T_{2,3,5} \cong \text{PSL}(2,\mathbb F_5)
\]
all with non-abelian image.

For the Euclidean cases, $T_{2,3,6}$ surjects $T_{2,3,3}$,  $T_{2,4,4}$ surjects the non-cyclic group $\mathbb Z/2 \mathbb Z \times \mathbb Z/2 \mathbb Z$, and $T_{3,3,3}$ surjects the non-cyclic group $\mathbb Z/3 \mathbb Z \times \mathbb Z/3 \mathbb Z$.  It remains to consider the triangle group $T_{2,2,m}$ with $m\geq 3$ odd as otherwise $T_{2,2,m}$ surjects  $\mathbb Z/2 \mathbb Z \times \mathbb Z/2 \mathbb Z$.  The triangle group $T_{2,2,m}$ is isomorphic to the dihedral group 
\[ D_{2m}=\langle s_1,s_2 \mid s_1^2=s_2^{m}, s_1s_2=s_2^{-1}s_1\rangle\] 
by $x\mapsto s_1$, $xy \mapsto s_2$. Let $p$ be any prime divisor of $m$,  $\pi$  a prime in $\Z[i]$ lying over $p$, and  $\mathbb F =  \mathbb \Z[i]/\pi$. Define $\rho:T_{2,2,m} \rightarrow \PSLTF$ by 

\[ x\mapsto \pm \mat{i}{0}{0}{-i}, \  y \mapsto \pm \mat{i}{i}{0}{-i}, \  \text{ so that } \ xy \mapsto \pm \mat{1}{1}{0}{1} \ \text{ and } \ yx \mapsto  \pm \mat{1}{-1}{0}{1} \]

Then we have that $\rho(x)$ and  $\rho(y)$ are order 2 and $\rho(xy)$ is order $p$. 
We observe that  $\rho(yx)\neq \rho(xy)$ as otherwise $\pi \mid 2i$ and so $p=2$ divides $m$ which is not the case.  We conclude that $T_{2,2,m}$ is isomorphic to a non-abelian subgroup of $\PSLTF$ where $|\mathbb F| \leq m^2$.
\end{proof}

We point out that the characteristic of $|\mathbb F|$ in the above proof is at most $n_3$.

\section{Degree bounds and  triangulations of $3$--manifolds}\label{sect:degreebounds}

This section describes some basic properties of the representations of closed manifold groups.  First, note that given a triangulation of a closed $3$--manifold $M$ with $t$ tetrahedra and $v$ vertices, there is a 1-vertex triangulation of the same manifold with $t$ or fewer tetrahedra \cite[Theorems 5.5, 5.6, and 5.14]{jaco20030}. We point out that if $M\cong \Sth$, then there is 1-vertex triangulation of $\Sth$ with one tetrahedron.

The following proposition  combines the arguments worked out in \cite[Theorem 6.12 and Exercise 6.3]{armstrong2013basic} with the existence of a 1-vertex triangulation for $M$. 
 
\begin{proposition}\label{prop:triangulation_facts}
Let $M$  be a triangulated, closed, orientable $3$--manifold admitting a triangulation with $t$ tetrahedra. Then there exists a (not necessarily minimal) presentation of $\pi_1(M)$ with at most $t+1$ generators and at most $2t$ relations each of length at most $3$. 
\end{proposition}

The presentation obtained in this manner has the advantage that the relations are the same length. We will use $t$, the number of tetrahedra in a triangulation $T$ as our measure of the complexity of $M$.    We are interested in non-abelian representations.  As such, we are interested in non-elementary subgroups of $\PSLTC$.  We call a subgroup $G$ of $\PSLTC$ {\em elementary } if the action of $G$ on $\CC$ has at least one point with finite orbit. Otherwise, we say it is {\em non-elementary.} Non-elementary groups are non-abelian.

We now bound the degree of the trace field of certain $\rho(G)<\PSLTC$ by a function of $t$. We will then apply this bound to our triangle group representations $\varrho$ of Seifert fiber space groups from  Definition~\ref{defn:varrho} in Section~\ref{sect:triangleReps}.  As in the proof of Proposition~\ref{prop:trianglePSL}, our proof below makes use of the $\PSLTC$ character variety of the (finitely presented) group $G$, which we define as the set 
\[
\{ \chi_{\rho} \mid \rho:G\rightarrow \PSLTC\}
\]
where the character is the function  $\chi_{\rho}:G\rightarrow C$ defined by $\chi_{\rho}(g) = (\text{tr}(\rho(g)))^2$. This set is a complex affine algebraic set defined over $\mathbb Q$. 

\begin{lemma}\label{lem:trace_field_bounds}
Let $G$ be a group with at most $t+1$ generators and at most $2t$ relations each of length at most $3$.  Let $X$ be an irreducible component of the  $\PSLTC$ character variety of $G$ of dimension 0, such that the representations corresponding to $X$ are non-elementary. For any representation $\rho:G\rightarrow \PSLTC$ with $\chi_{\rho}\in X$, the degree of the trace field of $\rho(G)$ is bounded above by $2^{t-1}3^{6t}$. 

\end{lemma}

\begin{proof}

We will label the generators of $\rho(G)$ as $\xi_k$ and we will assume that the first two generators in our presentation generate a non-elementary subgroup (that is they do not commute and do not have a common fixed point).   After conjugating the image of $\rho(G)$ if necessary, we may assume that each generator  is of the form 
\[ \gamma_1 = \pm \begin{pmatrix} \alpha_1 & 1\\0 & \alpha_1^{-1} \end{pmatrix}, 
\gamma_2= \pm  \begin{pmatrix} \alpha_2 & 0\\ \gamma_2 & \alpha_2^{-1} \end{pmatrix}, 
\text{ and } \gamma_k = \pm \begin{pmatrix} \alpha_k & \beta_k\\ \gamma_k & \delta_k \end{pmatrix}\]
for $k>2$.  The field  $L=\Q( \{\alpha_k,\beta_k,\gamma_k,\delta_k\}_{k\geq 1})$ is necessarily a number field, since $\dim_{\mathbb C} X=0$ and is the field of definition of $\rho$.  Let $R(G)$ be the (restricted) representation variety of $G$ determined by representations of this type. That is, $R(G)$ is the set of all such representations of $G$.  The set  $R_X(G)= \{\rho\in R(G):  \chi_{\rho} \in X\}$   is zero-dimensional by the triple transitively of $\PSLTC$ and the fact that $X$ has dimension $0$.

The  variety $R_X(G)$ is determined by the determinant equations and the relation equations. There are $t-1$ equations for the determinants, $\alpha_k \delta_k - \beta_k \gamma_k - 1=0$, which have degree 2. Each of the $t+1$ relations gives 4 equations, one per matrix entry.  We only need $3$ of these equations per relation, since the fourth entry is determined by the fact that all matrices have determinant $1$. Thus, there are at most $6t$ equations  we need to use coming from the relations and $t-1$ equations of degree 2 coming from the determinant condition.

We conclude that there are at most  $t-1$ degree 2 equations and $6t$ degree 3 equations which cut out a 0 dimensional variety whose entries are in the number field $L$.   
By repeatedly applying resultants then appealing to Bezout's theorem, $[L:\Q] \leq   2^{t-1}  3^{6t}$.

The degree of the trace field of $\rho(G)$ is also  bounded above by $2^{t-1}3^{6t}$ since $\text{tr}(\rho(G))\subset  L$.
\end{proof}

\section{An algorithm to distinguish small Seifert fiber spaces from lens spaces}\label{sect:sfs_distinguish}

In this section, we provide a method for distinguishing a small, non-cyclic  Seifert fiber space $M$ from a lens space and $\SoneXStwo$.  If an orientable Seifert fiber space is not small, then it is toroidal.  By work of Scott, Jaco and Shalen summarized in Proposition~\ref{prop:surjectstriangle}, the  base orbifold for $M$ is of the form $S^2(n_1,n_2,n_3)$. In particular, a Seifert fiber space with base orbifold $\mathbb{RP}^2(p)$ also admits a Seifert fibration with base orbifold $S^2(2,2,p)$.
By  \cite[Theorem 4]{HarawayHoffman2019complexity}, such a  toroidal manifold can be distinguished from  spaces with cyclic fundamental group in polynomial time. As discussed in Section \ref{sub:storing_finite_groups}, if the Seifert fiber space is small and non-cyclic the certificate required in  Theorem~\ref{main_thm} is a non-abelian representation to $\PSLTF$ or  non-cyclic homology, which can be verified in polynomial time.   Specifically, by Proposition~\ref{prop:surjectstriangle} if $M$ is a small, prime, non-cyclic Seifert fiber space then $\pi_1(M)$ surjects a triangle group $T_{n_1,n_2,n_3}$.  By Lemma~\ref{lemma:Tsurjects} if $\gcd(n_1,n_2,n_3)>1$ then $M$ has non-cyclic homology.  Theorem~\ref{thm:using_Linnik} guarantees a non-abelian representation to $\PSLTF$ if the base orbifold is hyperbolic and $\gcd(n_1,n_2,n_3)=1$.  In the remaining case where $\gcd(n_1,n_2,n_3)=1$ and the base orbifold is non-hyperbolic Proposition~\ref{prop:non_abelian} guarantees bounded non-cyclic homology or a non-abelian representation to $\PSLTF$ for $\mathbb F$ small.  Below we make these statements precise.

\begin{algorithm}[ht]
\caption{\label{algo:sfs_distinguish_from_Lpq} Distinguish $M$, an orientable, small, non-cyclic  Seifert fiber space 
from a lens space}
\begin{algorithmic}
\STATE {\bf Input:} A triangulation of $M$, a small, non-cyclic Seifert fiber space, into  $t$ tetrahedra. 
Note that  the base orbifold is of the form $\Stw(n_1,n_2,n_3)$.
\STATE{\bf Step 1:}  Compute the homology of $M$ and check if  $H_1(M)$ is cyclic. If not, return the homology. 
\STATE{\bf Step 2:} ($H_1(M)$ is cyclic.) Enumerate primes $\{p\}$ until $\pi_1(M)$ admits a non-trivial, non-abelian  $\PSLTFpSq$ representation.

\end{algorithmic}

\end{algorithm}

In Step 2 of Algorithm \ref{algo:sfs_distinguish_from_Lpq}, four cases are possible:  the base orbifold is hyperbolic, the base orbifold is $\Stw(2,2,m)$, the base orbifold is Euclidean or the base orbifold is spherical but not $\Stw(2,2,m)$. In the latter three cases, Proposition \ref{prop:non_abelian} determines bounds on size of the non-cyclic representation. These bounds are related to the triangle group data. Parametrizing these bounds in terms of $t$ and dealing with the hyperbolic case will occupy the remainder of this paper.

Long and Reid  (see \cite[Proof of Theorem 1.2]{long1998simple}) show the existence of non-trivial, non-abelian $\PSLTFp$ representations for infinitely many primes $p$  via quaternion algebras. Here $p$ is any split prime in the trace field of $M$.  Their method does not require prescreening by the homology check in Step 1, and does not bound the size of $p$.  Theorem~\ref{thm:SFS_certs} below provides a direct proof that such a $\PSLTFp$ or  $\text{PSL}(2,\mathbb F_{p^2})$ representation exists in this setting and  bounds  the size of $p$.  To use the non-trivial $\PSLTFp$ or  $\text{PSL}(2,\mathbb F_{p^2})$ representation as a suitable certificate  it is enough that $\log(p)$ is bounded by a globally defined polynomial parametrized by $t$ as $\log(p)$ (plus some constant) bounds the bit-size of the entries of those matricies (see Section \ref{sub:storing_finite_groups}).  


We begin with a few lemmas.  First, we give the following bound on $n$ in terms of $\phi(n)^2$.

\begin{lemma}\label{lemma:eulerphi}
 Let $n>2$ be an integer.  Then $\phi(n)<n\leq \phi(n)^2$ where $\phi$ denotes Euler's function.
 \end{lemma}

 \begin{proof}
By definition  $\phi(n) = n \prod_{p\mid n} (1-\tfrac1p)$ where the product is over distinct primes dividing $n$. It suffices to show that $n\leq \phi(n)^2$. Since $\phi(n)$ is multiplicative it's enough to show that for an integer $k\geq 1$ and a prime $p$ that $p^k \leq  \phi(p^k)^2$ unless $p=2$ and $k=1$.  Since $\phi(p^k)^2 = p^{2k-2}(p-1)^2$ it is elementary to see that if $p\neq 2$, then 
\[
p^k \leq p^{2k-2}(p-1)^2 = \phi(p^k)^2.
\]
If $p=2$ this inequality holds for $k>1$.
\end{proof}

We are now ready to prove our main technical result. The proof of our main theorem, Theorem~\ref{main_thm}, follows after. The key ingredients of this theorem have already been established, and we now reconcile the bounds discussed above in Proposition~\ref{prop:trianglePSL} and Lemma~\ref{lem:trace_field_bounds} to establish a connection between the number of tetrahedra in a triangulation of a small, non-cyclic Seifert fiber space and the size of a small congruence quotient.

\begin{theorem}\label{thm:SFS_certs}
Assume that $M$ is an orientable, closed, small, non-cyclic  Seifert fiber space
 admitting a triangulation with $t$ tetrahedra. Then either $\pi_1(M)$ surjects a non-cyclic abelian group  of order at most $2^{4t}3^{24 t}$ or $\pi_1(M)$ surjects a (non-trivial) non-abelian subgroup of $\PSLTF$ where $|\mathbb F| < c(2^{20t}3^{120t})$ for some effectively computable $c>0$. 
\end{theorem}

\begin{proof} 
If $M \cong \mathbb{RP}^3\# \mathbb{RP}^3$, then $\pi_1(M)$ surjects $\Z/2\Z \times \Z/2\Z$ and the claim is satisfied. Otherwise, we assume $M$ is prime.
As usual, we use $t$ to denote the number of tetrahedra in our triangulation $\TT$.  Since $M \not\cong \Sth$,  by \cite[Theorems 5.5, 5.6, and 5.14]{jaco20030} either $\TT$ is a 1-vertex triangulation of there is a 1-vertex triangulation with at most $t$ tetrahedra. Therefore we can assume that $\TT$ is a 1-vertex triangulation.


By the discussion in Section~\ref{sect:triangleReps}, the base orbifold is of the form $\Stw(n_1,n_2,n_3)$ and  $\pi_1(M)$ surjects $T_{n_1,n_2,n_3}$.  Let $\ell =2\mathrm{lcm}(n_1,n_2,n_3)$ and $d=\mathrm{gcd}(n_1,n_2,n_3)$.

First we assume that the base orbifold is hyperbolic. We begin by showing that  $\ell $ is bounded above by  $2^{2t}3^{12t}$. By Proposition~\ref{prop:trianglePSL},   $K=\Q(c_{\ell},r)$ is the field of definition  of the representation $\varrho$ given in Definition~\ref{defn:varrho} and  $\Q(c_{\ell})$ is the trace field with $[\Q(c_{\ell}):\Q]=\tfrac12 \phi(\ell)$.   Moreover, by Proposition~\ref{prop:trianglePSL} this representation is rigid and the corresponding component of the character variety has dimension 0.  By Lemma~\ref{lem:trace_field_bounds}, 
\[ [\Q(c_{\ell}):\Q]\leq 2^{t-1} 3^{6t}.\]  By  Lemma~\ref{lemma:eulerphi}, $\tfrac12 \ell^{\frac12}  \leq \tfrac12 \phi(\ell)$ and since $\tfrac12 \phi(\ell)= [\Q(c_{\ell}):\Q]$ we have 
\[
\tfrac12 \ell^{\frac12}  \leq \tfrac12 \phi(\ell)= [\Q(c_{\ell}):\Q] \leq 2^{t-1} 3^{6t}
\]
which implies our claimed bound for $\ell$. 

If $d>1$,  then by Lemma~\ref{lemma:Tsurjects}, $T_{n_1,n_2,n_3}$ surjects $\Z/d\Z \times \Z/d\Z$.  Since this is abelian we conclude that the integral homology of $\pi_1(M)$ contains a subgroup isomorphic to $Z/d\Z \times \Z/d\Z$. By the above, $\ell \leq 2^{2t} 3^{12t}$ and since $d<\ell $ the statement follows.

If  $d=1$, by the above   $\ell \leq 2^{2t}3^{12t}$. By Theorem \ref{thm:using_Linnik}  there is an effectively computable $c'>0$ and a field $\mathbb{F}$ with $| \mathbb F|  \leq c' \ell^{10}$ such that  $T_{n_1,n_2,n_3}$ has a non-abelian representation into $\PSLTF$.  The statement follows from combining these bounds. 
 
It remains to consider the case when the base orbifold is not hyperbolic. By Proposition \ref{prop:non_abelian} unless the base orbifold is $S^3(2,2,m)$ for $m\geq 3$ odd, $\pi_1(M)$ has a non-abelian representation into $\text{PSL}(2,\mathbb F)$ with $|\mathbb F|\leq 9$ or surjects a non-cyclic abelian group of order at most 9. In the case   when the base orbifold is $S^3(2,2,m)$ for $m\geq 3$ odd, $\pi_1(M)$ has a non-abelian representation into $\text{PSL}(2,\mathbb F)$  with $|\mathbb F|\leq m^2$ (again by  Proposition \ref{prop:non_abelian}).   It suffices to show that in this case $m\leq 3^{6t}$ so that $|\mathbb F| \leq 3^{12t}$. 
 The image of this  representation is isomorphic to the dihedral group of order $2m$ and so $\pi_1(M)$ contains an index 2 subgroup. Therefore $M$ has a double cover $\tilde M$, such that $|\pi_1(\tilde M)^{ab}|$ surjects $\Z/m\Z$.   Since $\tilde{M}$ is a double cover, and $M$ has a triangulation with $t$ tetrahedra, $\tilde M$ has a triangulation with at most $2t$ tetrahedra. Therefore $\tilde M$ has a 1-vertex triangulation with at most $2t$ tetrahedra.   Therefore, by Proposition~\ref{prop:triangulation_facts} we can construct a presentation for $\pi_1(\tilde{M})$ with at most $2t+1$ generators and $4t$ relations each of length 3. We note that $\pi_1(M)$ and $\pi_1(\tilde{M})$ are elliptic manifolds, and therefore have finite fundamental group and finite homology. Using Proposition \ref{prop:homology}, for example, we conclude that  $|\pi_1(\tilde M)^{ab}| \leq 3^{4t}$, so that $m\leq 3^{4t}$.  
\end{proof}



We now prove the main theorem. 

\begin{proof}[Proof of Theorem \ref{main_thm}]
Let $t$ be the number of tetrahedra in a triangulation of $M$. 
By Proposition~\ref{prop:nonorientability} there is a polynomial time algorithm in $t$ to determine whether $M$ is orientable or non-orientable.  Therefore we can distinguish non-orientable $3$--manifolds from lens spaces.  We now assume that  $M$ is orientable. 
 The only two  orientable reducible Seifert fiber spaces are $\mathbb{RP}^3\# \mathbb{RP}^3$ and $\SoneXStwo$.  Both {\sc $\SoneXStwo$ recognition} and  {\sc $\mathbb{RP}^3\# \mathbb{RP}^3$ recognition}  problems lie in NP (see \cite[Theorem 3]{ivanov2008}). These certificates can then be used to certify that $M$ is not a lens space. (More concretely, if $M\cong \mathbb{RP}^3\# \mathbb{RP}^3$, $H_1(M)\cong \Z/2\Z \times \Z/2\Z$.)  Therefore it suffices to consider orientable, prime, non-cyclic $M$. 
 
 If $M$ is not small, then it is toroidal and we appeal to \cite[Theorem 4]{HarawayHoffman2019complexity}, which certifies that $M$ is not a lens space via a polynomial time verifiable certificate with input of size $t$. Finally, we may assume that $M$ is orientable, small, prime, and non-cyclic. We combine the results from Theorem \ref{thm:SFS_certs} and Proposition \ref{prop:cert_non_abelian} to prove that such a certificate exists.\end{proof}

Corollary \ref{cor:SoneXStwo} immediately follows from the arguments stated above.

\begin{remark}\label{rem:concluding_rem}
Although not Seifert fiber spaces, we point out that only one Sol manifold has cyclic homology (see \cite{scott1983geometries} for example). In the one case of cyclic homology, the double cover has homology $\Z/5\Z \times \Z$. So if $M$ admits Sol geometry, the homology of the manifold or its double cover serves as a certificate that $M$ is neither a lens space nor $\SoneXStwo$.
\end{remark}

In light of our results, to prove part 1) of Conjecture \ref{conj:s3andLpq} it is enough to show that {\sc $\Sth$ recognition} lies in coNP for hyperbolic integral homology spheres. We anticipate that  the extension to {\sc lens space recognition} will be considerably more difficult.  However our initial discussions on the Seifert fiber space case began by seeking a sufficiently small, non-trivial  $\PSLTF$ representation. We later realized our methods apply in the more general setting as described above.
 

\bibliographystyle{plain}
\bibliography{SFSbib}

\end{document}